\documentclass[12pt]{article}

\usepackage{amsmath,amsthm,amsfonts,amssymb,amscd,cite}

\allowdisplaybreaks[4]

\newtheorem {theorem} {Theorem}
\newtheorem {conjecture}{Conjecture}
\newtheorem {corollary}{Corollary}

\newtheorem {claim}{Claim}

\begin{document}

\title{Hamiltonicity of $1$-tough $(P_2\cup kP_1)$-free graphs}

\author{Leyou Xu\footnote{Email: leyouxu@m.scnu.edu.cn}, Chengli Li\footnote{E-mail: lichengli@m.scnu.edu.cn}, Bo Zhou\footnote{Corresponding author.
E-mail: zhoubo@m.scnu.edu.cn}\\
School of Mathematical Sciences, South China Normal University\\
Guangzhou 510631, P.R. China}

\date{}
\maketitle

\begin{abstract}
Given a graph $H$, a graph $G$ is $H$-free if $G$ does not contain $H$ as an induced subgraph.
For a positive real number $t$,  a non-complete graph $G$ is said to be $t$-tough if for every vertex cut $S$ of $G$, the ratio of $|S|$ to the number of components of $G-S$ is at least $t$. A complete graph is said to be $t$-tough for any $t>0$.
Chv\'{a}tal's toughness conjecture, stating  that there exists a constant $t_0$ such that every $t_0$-tough graph
with at least three vertices is Hamiltonian,   is still open in general. Chv\'{a}tal and Erd\"{o}s  \cite{CE} proved  that, for any  integer $k\ge 1$,
every $\max\{2,k\}$-connected $(k+1)P_1$-free graph on at least three vertices is Hamiltonian.
Along the  Chv\'{a}tal-Erd\"{o}s theorem,
Shi and Shan \cite{SS} proved  that,  for any  integer $k\ge 4$, every $4$-tough $2k$-connected $(P_2\cup kP_1)$-free graph with at least three vertices is Hamiltonian, and furthermore, they proposed a conjecture that for any  integer $k\ge 1$, any $1$-tough  $2k$-connected  $(P_2\cup kP_1)$-free graph is Hamiltonian.
In this paper, we  confirm the conjecture, and furthermore, we show that if $k\ge 3$, then the condition `$2k$-connected' may be weakened to be `$2(k-1)$-connected'.
As an immediate consequence, for any integer $k\ge 3$, every $(k-1)$-tough $(P_2\cup kP_1)$-free graph is Hamiltonian. This improves the result of Hatfield and Grimm \cite{HG}, stating that every $3$-tough $(P_2\cup 3P_1)$-free graph is Hamiltonian.
\end{abstract}


\section{Introduction}

Let $G$ be a graph with vertex set $V(G)$ and edge set $E(G)$.
A graph $G$ is Hamiltonian if there exists a cycle containing each vertex of $G$.
For a given graph $H$, a graph $G$ is called $H$-free if $G$ does not contain $H$ as an induced subgraph.

For vertex disjoint graphs $H$ and $F$, $H\cup F$ denotes the disjoint union of graphs $H$ and $F$.
A linear forest is a graph consisting of disjoint paths.
As usual, $P_n$ denotes the path on $n$ vertices. For positive integer $k$ and $\ell$, $kP_{\ell}$ denotes the linear forest consisting of $k$ disjoint copies of the path $P_{\ell}$.

For a positive integer $k$, a connected graph $G$ is said to be $k$-connected if any deletion of at most $k-1$ vertices on $G$ also results in a connected graph.

For a graph $G$ with $S\subset V(G)$, denote by $G[S]$ the subgraph of $G$ induced by $S$.
Let $G-S=G[V(G)-S]$.
The number of components of $G$ is denoted by $c(G)$.

The toughness of a graph $G$, denoted by $\tau(G)$, is defined as
\[
\tau(G)=\min\left\{\frac{|S|}{c(G-S)}: S\subseteq V(G), c(G-S)\ge 2\right\}
\]
if $G$ is not a complete graph and $\tau(G)=\infty$ otherwise.
For a positive real number $t$,
a graph $G$ is called $t$-tough if $\tau(G)\ge t$, that is, $|S|\ge t\cdot c(G-S)$ for each $S\subseteq V(G)$ with $c(G-S)\ge 2$.
The concept of toughness of a graph was introduced by Chv\'{a}tal \cite{Ch}.
Clearly, every  Hamiltonian graph is $1$-tough, but the converse is not true.
Chv\'{a}tal \cite{Ch} proposed the following conjecture, which is known as Chv\'{a}tal's toughness conjecture.

\begin{conjecture}[Chv\'{a}tal] \cite{Ch} \label{Or}
There exists a constant $t_0$ such that every $t_0$-tough graph
with at least three vertices is Hamiltonian.
\end{conjecture}

Bauer, Broersma and Veldman \cite{BBV} showed that $t_0\ge \frac{9}{4}$ if it exists.
Conjecture \ref{Or}  has been confirmed  for a number of special classes of graphs
\cite{BBS,BHT,BPP,BXY,CJKJ,DKS,GS,HG,KK,LBZ,OS,Shan0,Sh,Shan,SS,Wei}.
For example, it has been confirmed for graphs with forbidden (small) linear forests, such as $1$-tough $R$-free graphs with $R=P_3\cup P_1, P_2\cup 2P_1$ \cite{LBZ},
$2$-tough $2P_2$-free graphs \cite{BPP,Sh,OS}, $3$-tough $(P_2\cup 3P_1)$-free graphs  \cite{HG}, $7$-tough $(P_3\cup 2P_1)$-free graphs \cite{GS} and $15$-tough $(P_3\cup P_2)$-free graphs \cite{Shan}. Though great efforts have been made,  it remains open. 

If connectivity is considered, there is a classic result, due to Chv\'{a}tal and Erd\"{o}s \cite{CE}.

\begin{theorem}[Chv\'{a}tal and Erd\"{o}s] \cite{CE} \label{coe}
For any integer $k\ge 1$, every $\max\{2,k\}$-connected $(k+1)P_1$-free graphs on at least three vertices is Hamiltonian.
\end{theorem}

Note that $k$-connected $(k+1)P_1$-free graphs must be $1$-tough and that constant connectivity condition cannot guarantee the existence of a Hamiltonian cycle in $(P_2\cup kP_1)$-free graphs.
Supporting Chv\'{a}tal's toughness conjecture, Shi and Shan \cite{SS} and Hatfield and Grimm \cite{HG} established the following interesting results.  

\begin{theorem}[Shi and Shan]  \cite{SS}
For any integer $k\ge 4$,
every $4$-tough  $2k$-connected  $(P_2\cup kP_1)$-free graph is Hamiltonian.
\end{theorem}

\begin{theorem}[Hatfield and Grimm] \cite{HG} \label{o}
Every $3$-tough $(P_2\cup 3P_1)$-free graph is Hamiltonian.
\end{theorem}

Shi and Shan \cite{SS} proposed the following conjecture.

\begin{conjecture}[Shi and Shan] \cite{SS} \label{PF}
Let $k\ge 4$ be an integer.
Let $G$ be a $1$-tough  $2k$-connected $(P_2\cup kP_1)$-free graph.
Then $G$ is Hamiltonian.
\end{conjecture}

In this paper, we show that Conjecture \ref{PF} is true by showing the following result.  

\begin{theorem}\label{xu}
For any  integer $k\ge 1$,
every  $1$-tough  $2k$-connected $(P_2\cup kP_1)$-free graph is Hamiltonian.
\end{theorem}

Furthermore,  we show the following stronger result.

\begin{theorem}\label{li}
Let $k$ be an integer with $k\ge 3$. Every $1$-tough  $(2k-2)$-connected $(P_2\cup kP_1)$-free graph  is Hamiltonian.
\end{theorem}



Theorems \ref{xu} and \ref{li} echo the Chav\'{a}tal-Erd\"{o}s  theorem \cite{CE} (Theorem \ref{coe}).
Note that a non-complete $(k-1)$-tough graph must be $(2k-2)$-connected for $k\ge 3$. An immediate consequence of Theorem \ref{li} is as follows, from which we also have Theorem \ref{o} due to Hatfield and Grimm \cite{HG}.

\begin{corollary} \label{zhong}
For any integer $k\ge 3$, every $(k-1)$-tough $(P_2\cup kP_1)$-free graph is Hamiltonian.
\end{corollary}

\section{Preliminaries}

We introduce some notations.

For $v\in V(G)$,  $N_G(v)$ denotes the neighborhood of $v$ in $G$.
For $v\in V(G)$ and a subgraph $F$ of $G$, let
$N_F(v)=N_G(v)\cap V(F)$. For $S\subseteq V(G)$, $N_F(S)=\bigcup_{v\in S} N_F(v)$.
If $H$ is a subgraph of $G$, then we write $N_F(H)$ for $N_F(V(H))$.

Let $C$ be an oriented cycle, where the orientation is always clockwise. For $u\in V(C)$,
denote by $u^{+1}$ the immediate successor of $u$ and $u^{-1}$ the immediate predecessor of $u$ on $C$.
For an integer  $\ell\ge 2$,
denote by $u^{+\ell}$ the immediate successor of $u^{+(\ell-1)}$ and $u^{-\ell}$ the immediate predecessor of $u^{-(\ell-1)}$ on $C$.
For convenience, we write  $u^{+}$ for  $u^{+1}$ and $u^{-}$ for  $u^{-1}$.
For $S\subseteq V(C)$, let $S^+=\{u^+:u\in S\}$. 
For $u,v\in V(C)$,  $u\overrightarrow{C}v$ denotes the segment of $C$ from $u$ to $v$ which follows the orientation of $C$, while $u\overleftarrow{C}v$ denotes the opposite segment of $C$ from $u$ to $v$.
Particularly, if $u=v$, then $u\overrightarrow{C}v=u$ and $u\overleftarrow{C}v=u$.

For a graph $G$ with $u,v\in V(G)$, a $(u,v)$-path is a path from $u$ to $v$ in $G$.

\section{Proof of Theorem \ref{xu}}



\begin{proof}[Proof of Theorem \ref{xu}]
Suppose to the contrary that $G$ is a $1$-tough $2k$-connected $(P_2\cup kP_1)$-free graph but $G$ is not Hamiltonian. Then $G$ is not complete.
As $G$ is $2k$-connected, there are cycles in $G$.
Let $C$ be a longest cycle in $G$. As $G$ is not Hamiltonian, $V(G)\setminus V(C)\ne \emptyset$.
Observe that  $N_C(H)\ne \emptyset$ for any component $H$ of $G-V(C)$.

\begin{claim} \label{Claim1}
For any component $H$ of $G-V(C)$,
$N_C(H)^+$ is an independent set, and $N_C(H)\cap N_C(H)^+=\emptyset$.
\end{claim}

\begin{proof}
Suppose that $N_C(H)^+$ is not independent for some component $H$ of $G-V(C)$.
Let $N_C(H)=\{u_1,\dots,u_t\}$.
Then  $u_i^+u_j^+\in E(G)$ for some $i$ and some $j$ with $1\le i<j\le t$.
Let $u_i'$ be a neighbor of $u_i$ in $H$ and $u_j'$ a neighbor  $u_j$ in $H$.
As $H$ is connected, there is a $(u_j',u_i')$-path $P$ in $H$.
Then
\[
u_i\overleftarrow{C}u_j^+u_i^+\overrightarrow{C}u_ju_j'Pu_i'u_i
\]
is a cycle of $G$ longer than $C$, a contradiction.
So $N_C(H)^+$ is an independent set of $G$. 

As $N_C(H)^+$ is an independent set of $G$, we have $N_C(H)\cap N_C(H)^+= \emptyset$.
\end{proof}

\begin{claim} \label{Claim2}
Every component of $G-V(C)$ is trivial.
\end{claim}

\begin{proof}
Suppose to the contrary that there exists a nontrivial component $H$ of $G-V(C)$.
Then $H$ contains an edge $uv$.
By Claim \ref{Claim1},
$N_C(H)\cap N_C(H)^+=\emptyset$, so $G-N_C(H)$ is not connected, and
$N_C(H)$ is a vertex cut of $G$.
As $G$ is $2k$-connected, we have $|N_C(H)^+|=|N_C(H)|\ge 2k$.
By Claim \ref{Claim1}, $N_C(H)^+$ is an independent set of $G$ and then the graph  $G\left[N_C(H)^+\cup \{u,v\}\right]$ contains exactly one edge $uv$ and so it contains $P_2\cup 2kP_1$ as an induced subgraph. Thus,   $G$  contains $P_2\cup kP_1$ as an induced subgraph, a contradiction.
\end{proof}

Let $x\in V(G)\setminus V(C)$ and $N_C(x)=\{x_1,\dots,x_t\}$. By Claim \ref{Claim2}, $x$ is an isolated vertex of $G-V(C)$.  So, by Claim \ref{Claim1} and the fact that  $G$ is $2k$-connected,
$t\ge 2k$.
For $i=1,\dots,t$,
denote by $S_i$  the vertex set of the segment $x_i^+\overrightarrow{C}x_{i+1}^-$ of $C$ from $x_i^+$ to $x_{i+1}^-$,  where $x_{t+1}=x_1$.
As $x_i$ and $x_{i+1}$ are not consecutive vertices on $C$ by Claim \ref{Claim1},  $|S_i|\ge 1$.

\begin{claim} \label{Claim3}
For $i=1,\dots,t$,
$|S_i|$ is odd and $N_C(x_i^{+j})\cap N_C(x)^+=\emptyset$
if $ 1\le j\le |S_i|$ with $j\equiv 1~(\bmod~2)$.
\end{claim}

\begin{proof} Firstly, we prove the second part.

Assume that $i=1$ as the argument applies also to the case $i=2,\dots, t$.

Take an arbitrary $X\subseteq N_C(x)^+$ with $|X|=2k$ and $x_1^+\in X$. We will show that
\begin{equation}\label{Xu}
|N_C(x_1^{+j})\cap X|\begin{cases}
=0 & \mbox{if $j$ is odd},\\
\ge k+2 &  \mbox{if $j$ is even}
\end{cases}
\end{equation}
by induction on $j$ for integers $j=1,\dots, |S_1|$.
If $j=1$, then by Claim \ref{Claim1}, $N_C(x)^+$ is an independent set of $G$, so  $N_C(x_1^{+1})\cap N_C(x)^+=\emptyset$, i.e., \eqref{Xu} follows for $j=1$. Suppose that
\eqref{Xu} is not true for $j=2$. Then $\left|N_C(x_1^{+2})\cap X\right|\le  k+1$. So $\left|(X\setminus N_C(x_1^{+2})) \cup \{x,x_1^+,x_1^{+2}\}\right|\ge k+2$, and $G\left[(X\setminus N_C(x_1^{+2})) \cup \left\{x,x_1^+,x_1^{+2}\right\}\right]$ contains exactly one edge $x_1^{+}x_1^{+2}$, so it contains
$P_2\cup kP_1$ as an induced subgraph,
a contradiction. Thus, \eqref{Xu} follows for $j=2$.

Let $j$ be an integer with $3\le j\le |S_1|$. Suppose that \eqref{Xu} holds for $j-1$.

Suppose that $j$ is odd. By the inductive hypothesis, $\left|N_C(x_1^{+(j-1)})\cap X\right|\ge k+2$.
Suppose that $N_C(x_1^{+j})\cap X\ne \emptyset$.
Then there exists some $x_r^+\in N_C(x_1^{+j})\cap X$.
If $\left|N_C(x_1^{+j})\cap X\right|\le k+1$, then $\left| (X\setminus N_C(x_1^{+j}))\cup \left\{x,x_1^{+j},x_r^+\right\}\right|\ge k+2$ and $G\left[  (X\setminus N_C(x_1^{+j}))\cup \left\{x,x_1^{+j},x_r^+\right\}\right]$ contains exactly one edge $x_1^{+j}x_r^+$, so it contains $P_2\cup kP_1$ as an induced subgraph, a contradiction.
Then $\left|N_C(x_1^{+j})\cap X\right|\ge k+2$ and hence
\begin{align*}
\left|N_C(x_1^{+(j-1)})\cap N_C(x_1^{+j}) \cap X\right|
&\ge \left|N_C(x_1^{+(j-1)})\cap X\right|+\left|N_C(x_1^{+j}) \cap X\right|-|X|\\
&\ge  k+2+k+2-2k>2.
\end{align*}
Assume that $x_p^+,x_q^+\in N_C(x_1^{+(j-1)})\cap N_C(x_1^{+j})\cap X$ with $1\le p<q\le t$.
If $p\ge 2$, then
\[
x_1^{+(j-1)}x_p^+\overrightarrow{C}x_qxx_p\overleftarrow{C}x_1^{+j}x_q^{+}\overrightarrow{C}x_1^{+(j-1)}
\]
is a cycle of $G$ longer than $C$, a contradiction.
So $p=1$, and then
\[
x_1^{+(j-1)}x_q^{+}\overrightarrow{C}x_1xx_q\overleftarrow{C}x_1^{+j}x_1^{+}\overrightarrow{C}x_1^{+(j-1)}
\]
is a cycle of $G$ longer  than $C$, also a contradiction.
Therefore, $N_C(x_1^{+j})\cap X=\emptyset$. This is \eqref{Xu}  for odd $j$.

Now suppose that $j$ is even.
By the inductive hypothesis,  $N_C(x_1^{+(j-1)})\cap X=\emptyset$.
If $\left|N_C(x_1^{+j})\cap X\right|\le k+1$, then $\left| (X\setminus N_C(x_1^{+j}))\cup \left\{x,x_1^{+j},x_1^{+(j-1)}\right\}\right|\ge k+2$, and 
$G\left[(X\setminus N_C(x_1^{+(j-1)}))\cup \left\{x,x_1^{+j},x_1^{+(j-1)}\right\} \right]$
contains exactly one edge $x_1^{+j}x_1^{+(j-1)}$, so it contains $P_2\cup kP_1$ as an induced subgraph, a contradiction.
So $\left|N_C(x_1^{+j})\cap X\right|\ge k+2$. This is \eqref{Xu}  for even $j$.

If $1\le j\le |S_i|$ with $j\equiv 1~(\bmod~2)$, then by \eqref{Xu},  $N_C(x_1^{+j})\cap X=\emptyset$
for any $X\subseteq N_C(x)^+$ with $|X|=2k$ and $x_1^+\in X$, so
\[
N_C(x_1^{+j})\cap N_C(x)^+=N_C(x_1^{+j})\bigcap \bigcup_{X\subseteq N_C(x)^+ \atop {|X|=2k \atop x_1^+\in X} }X=\emptyset.
\]
This proves the second part.

Secondly, we prove the first part.
Suppose  that $|S_1|$ is even.
By \eqref{Xu}, $\left|N_C(x_1^{+|S_1|})\cap X\right|\ge k+2$.
If $\left|N_C(x_2)\cap X\right|\le k$, then $G\left[X\setminus N_C(x_2)\cup \{x,x_2\}\right]$ contains exactly one edge $xx_2$ and so it contains $P_2\cup kP_1$ as an induced subgraph, a contradiction.
So $\left|N_C(x_2)\cap X\right|\ge k+1$ and then by similar argument as above, we have $\left|N_C(x_2)\cap N_C(x_1^{+|S_1|})\cap X\right|\ge 2$ and so
we may obtain a cycle of $G$ longer than $C$, a contradiction.
Therefore, $|S_1|$ is odd, as desired.
\end{proof}


Let $S'=\cup_{i=1}^tS_i'$, where $S_i'=\left\{x_i^{+j}\in S_i:j\equiv 1~(\bmod~2)\right\}$ for  $i=1,\dots, t$.
Suppose that there is an edge $uv$ in $G[S']$. By Claim \ref{Claim3},
$N_C(u)\cap N_C(x)^+=\emptyset$ and $N_C(v)\cap N_C(x)^+=\emptyset$. By Claim \ref{Claim1},
$N_C(x)^+$ is an independent set. So
$uv$ is the unique edge in $G\left[N_C(x)^+\cup\{u,v\}\right]$. Recall that $|N_C(x)^+|\ge 2k$.
So $G\left[N_C(x)^+\cup\{u,v\}\right]$ contains $P_2\cup kP_1$ as an induced subgraph, a contradiction. So $S'$ is an independent set of $G$.
By Claim \ref{Claim3}, $|S_i|$ is odd for each $i=1,\dots,t$, so \[
|V(C)|=2|S'|.\]

\begin{claim} \label{Claim4}
For any $y\in V(G)\setminus  V(C)$, $N_C(y)\cap S'=\emptyset$.
\end{claim}

\begin{proof}
The case  $y=x$ is obvious by the definition of $S'$.
Suppose that $y\ne x$.

Firstly, we show that
\[
N_C(y)\cap N_C(x)^+=\emptyset.
\]
Otherwise, $\left|N_C(y)\cap N_C(x)^+\right|\ge 1$.
If $\left|N_C(y)\cap N_C(x)^+\right|\ge 2$,
then \[
x_p\overleftarrow{C}x_q^+yx_p^+\overrightarrow{C}x_qxx_p
\]
is a cycle longer than $C$ for some $x_p^+,x_q^+\in N_C(y)$ with  $1\le p<q\le t$, which is a contradiction.
So $\left|N_C(y)\cap N_C(x)^+\right|=1$, say $y'\in N_C(y)\cap N_C(x)^+$.
By Claim \ref{Claim1}, $N_C(x)^+$ is an independent set of $G$.
Recall that $|N_C(x)^+|\ge 2k$.
Then $G\left[N_C(x)^+\cup \{y\}\right]$ contains exactly one edge $yy'$ and so it contains $P_2\cup kP_1$ as an induced subgraph, also a contradiction.
So $N_C(y)\cap N_C(x)^+=\emptyset$.

Now we show that $N_C(y)\cap S'= \emptyset$. Suppose that this is not true.
Then there is a vertex  $z\in N_C(y)\cap S'$.
Since $N_C(y)\cap N_C(x)^+=\emptyset$ and $N_C(x)^+\subseteq S'$, we have
 $z\notin N_C(x)^+$.
By Claim \ref{Claim3}, $z$ is not adjacent to any vertex in $N_C(x)^+$.
Then $G\left[N_C(x)^+\cup \{y,z\}\right]$  contains exactly one edge $yz$ and so it contains $P_2\cup  kP_1$ as an induced subgraph, a contradiction.
\end{proof}

Let $S=V(C)\setminus S'$.
Then $|S|=\frac{1}{2}|V(C)|$.
By Claim \ref{Claim4}, for any $y\in V(G)\setminus V(C)$, $y$ is not adjacent to any vertex in $S'$. So $(V(G)\setminus V(C))\cup S'$ is an independent set by Claim \ref{Claim2}. So $c(G-S)=|V(G)|-|V(C)|+|S'|>\frac{1}{2}|V(C)|=|S|$ and so
\[
\tau(G)\le \frac{|S|}{c(G-S)}<1,
\]
contradicting to the condition that $\tau(G)\ge 1$.
This completes the proof of Theorem \ref{xu}.
\end{proof}

\section{Proof of Theorem \ref{li}}

\begin{proof}[Proof of Theorem \ref{li}]
Suppose to the contrary that $G$ is a $1$-tough  $(2k-2)$-connected $(P_2\cup kP_1)$-free graph that is not Hamiltonian.
Then $G$ is not complete. As $G$ is $(2k-2)$-connected, there are cycles in $G$.
Let $C$ be a longest cycle in $G$. Then $V(G)\setminus V(C)\ne \emptyset$.
It is evident that  $N_C(H)\ne \emptyset$ for any component $H$ of $G-V(C)$.

By the same argument as in
Claim \ref{Claim1}, we have

\begin{claim} \label{c1}
For any component $H$ of $G-V(C)$,
$N_C(H)^+$ is an independent set, and $N_C(H)\cap N_C(H)^+=\emptyset$.
\end{claim}

%

For any component $H$ of $G-V(C)$, $N_C(H)\cap  N_C(H)^+=\emptyset$ by Claim \ref{c1}, and so $N_C(H)$ is a vertex cut of $G$.
As $G$ is $(2k-2)$-connected, $|N_C(H)^+|=|N_C(H)|\ge 2k-2$. Similarly as in
Claim \ref{Claim2}, we have

\begin{claim}\label{c2}
Every component of $G-V(C)$ is trivial.
\end{claim}

Let $x\in V(G)\setminus V(C)$ and $N_C(x)=\{x_1,\dots,x_t\}$. By Claim \ref{c2}, $x$ is an isolated vertex of $G-V(C)$.  So
$t=|N_C(x)|\ge 2k-2$.
For convenience, let $x_{t+i}=x_i$ for $i=1,\dots,t$.
For $i=1,\dots,t$,
denote by $S_i$  the vertex set of the segment $x_i^+\overrightarrow{C}x_{i+1}^-$ of $C$ from $x_i^+$ to $x_{i+1}^-$. 
By Claim \ref{c1}, $x_i$ and $x_{i+1}$ are not consecutive vertices on $C$, so    $|S_i|\ge 1$

\begin{claim} \label{c3}
For $i=1,\dots,t$ and $j=1,\dots,|S_i|$, $N_C(x)^+\cap N_C(x_i^{+j})=\emptyset$ if $j$ is odd and
$\left|N_C(x)^+\setminus N_C(x_i^{+j}) \right|\le k-2$ if $j$ is even.
\end{claim}

\begin{proof}
We  prove Claim \ref{c3} by induction on $j$ for $j=1,\dots,|S_i|$.

By Claim \ref{c1}, $N_C(x)^+$ is an independent set, so $N_C(x)^+\cap N_C(x_i^{+})=\emptyset$, i.e., Claim \ref{c3} follows for $j=1$.

Suppose $\left| N_C(x)^+\setminus N_C(x_i^{+2})\right|\ge k-1$. Let $V_i=( N_C(x)^+\setminus N_C(x_i^{+2})) \cup \{x,x_i^+,x_i^{+2}\}$. Then
$G[V_i]$ contains exactly one edge $x_i^+x_i^{+2}$, so it contains $P_2\cup kP_1$ as an induced subgraph, a contradiction.
Thus  $\left|  N_C(x)^+\setminus N_C(x_i^{+2})\right|\le k-2$, i.e.,  Claim \ref{c3} follows for $j=2$.

Suppose that $j$ be an integer with $3\le j\le |S_i|$ and that Claim \ref{c3} holds for $j-1$.

Suppose that $j$ is odd.
By induction assumption,
\[
\left| N_C(x)^+\setminus N_C\left(x_i^{+(j-1)}\right) \right|\le k-2.
\]
We want to show that $N_C(x)^+\cap N_C(x_i^{+j})= \emptyset$.
Suppose that $N_C(x)^+\cap N_C(x_i^{+j})\ne \emptyset$,
say $z\in N_C(x)^+\cap N_C(x_i^{+j})$.
If $\left|N_C(x)^+\setminus N_C(x_i^{+j})\right|\ge k-1$, then $G\left[(N_C(x)^+\setminus N_C(x_i^{+j}))\cup \{x,x_i^{+j},z\} \right]$
contains exactly one edge $x_i^{+j}z$ and so it contains $P_2\cup kP_1$ as an induced subgraph,  a contradiction.
Thus
\[
\left|N_C(x)^+\setminus N_C(x_i^{+j})\right|\le k-2,
\]
which implies that
\begin{align*}
&\quad \left|N_C(x_i^{+j})\cap N_C\left(x_i^{+(j-1)}\right)\cap N_C(x)^+\right|\\
&\ge
\left|N_C(x_i^{+j})\cap N_C(x)^+\right|+\left| N_C\left(x_i^{+(j-1)}\right)\cap N_C(x)^+\right|-\left| N_C(x)^+\right|\\
&\ge t-(k-2)+t-(k-2)-t\ge 2.
\end{align*}
So we may assume that $x_p^+,x_q^+\in N_C(x_i^{+j})\cap N_C\left(x_i^{+(j-1)}\right)\cap N_C(x)^+$ with $i+1\le p<q\le t+i$. 
Then
\[
x_i^{+(j-1)}x_p^+\overrightarrow{C}x_qxx_p\overleftarrow{C}x_i^{+j}x_q^{+}\overrightarrow{C}x_i^{+(j-1)}
\]
is a cycle of $G$ longer than $C$, a contradiction.
Therefore, $N_C(x)^+\cap N_C(x_i^{+j})=\emptyset$, proving  Claim \ref{c3} for odd $j$.

Suppose that $j$ is even.
By induction assumption, $N_C(x)^+\cap N_C\left(x_i^{+(j-1)}\right)=\emptyset$.
If $\left|N_C(x)^+\setminus N_C(x_i^{+j}) \right|\ge k-1$, then $G\left[(N_C(x)^+\setminus N_C(x_i^{+j}))\cup \{x,x_i^{+(j-1)},x_i^{+j} \}\right]$ contains exactly one edge $x_i^{+(j-1)}x_i^{+j}$, so it contains $P_2\cup kP_1$ as an induced subgraph, a contradiction.
So $\left|N_C(x)^+\setminus N_C(x_i^{+j}) \right|\le k-2$, proving Claim \ref{c3} for even $j$.
\end{proof}

For $i=1,\dots, t$, let $S_i'=\left\{x_i^{+j}\in S_i: j\equiv 1~(\bmod~2)\right\}$ and $S':=\cup_{i=1}^tS_i'$.

\begin{claim}\label{in}
$S'$ is an independent set.
\end{claim}

\begin{proof}

Suppose that  $S'$ is not independent.
Then there is an edge $uv$ in $G[S']$.  By Claim \ref{c3},
 $N_C(x)^+\cap N_C(u)=\emptyset$ and
 $N_C(x)^+\cap N_C(v)=\emptyset$.  By Claim \ref{c1}, $N_C(x)^+$ is an independent set. It thus follows that
$uv$ is the unique edge in $G\left[N_C(x)^+\cup\{u,v\}\right]$. As $k\ge 3$, $|N_C(x)^+|\ge 2k-2\ge k+1$.
So $G\left[N_C(x)^+\cup\{u,v\}\right]$ contains $P_2\cup kP_1$ as an induced subgraph, a contradiction.
\end{proof}

\begin{claim}\label{k-2}
For any $z\in V(C)\setminus (N_C(x)\cup S')$, $\left|S'\setminus N_C(z)\right|\le k-2$.
\end{claim}

\begin{proof}
Suppose that  $\left|S'\setminus N_C(z)\right|\ge k-1$ for some $z\in V(C)\setminus (N_C(x)\cup S')$.
As $z\notin N_C(x)$ and $z^-\in S'$, we have by Claim \ref{in} that $S'\setminus N_C(z)\cup \{z^-\}$ is an independent set. Then $G\left[ (S'\setminus N_C(z))\cup \{x,z,z^-\}\right]$ contains exactly one edge $zz^-$ and so it contains $P_2\cup kP_1$ as an induced subgraph, a contradiction.
Therefore $\left|S'\setminus N_C(z)\right|\le k-2$.
\end{proof}

\begin{claim}\label{xi}
For any $i=1,\dots,t$, $\left|N_C(x)^+\setminus N_C(x_i)\right|\le k-1$.
\end{claim}
\begin{proof}
If $\left|N_C(x)^+\setminus N_C(x_i)\right|\ge k$ for some $i=1,\dots,t$,
then by Claim \ref{c1}, $G\left[ (N_C(x)^+\setminus N_C(x_i))\cup \{x_i,x_i^+\}\right]$ contains exactly one edge $x_ix_i^+$ and so it contains $P_2\cup kP_1$ as an induced subgraph, a contradiction.
\end{proof}

\begin{claim}\label{even}
If $|S_i|$ is even for some $i=1,\dots,t$, then
\begin{enumerate}
\item[(i)] $t=2k-2$, \\
\item[(ii)] $x_{i+1}^+,\dots,x_{i+k-2}^+\notin N_C\left(x_i^{+|S_i|}\right)$,\\
\item[(iii)]  $x_{i+k}^+,\dots,x_{i+t}^+\notin N_C(x_{i+1})$.
 \end{enumerate}
\end{claim}

\begin{proof}
Let $x_p^+\in N_C\left(x_i^{+|S_i|}\right)$ such that $i+1\le p\le t+i$ and $p$ is as small as possible.
Let $x_q^+\in N_C(x_{i+1})$ such that $i+1\le q\le t+i$ and $q$ is as large as possible.
If $p<q$, then
\[
x_i^{+|S_i|}x_p^+\overrightarrow{C}x_qxx_p\overleftarrow{C}x_{i+1}x_q^+\overrightarrow{C}x_i^{+|S_i|}
\]
is a cycle of $G$ longer than $C$, a contradiction.
So $p\ge q$.

As $|S_i|$ is even, we have by Claim \ref{c3} that $\left|N_C(x)^+\setminus N_C(x_i^{+|S_i|}) \right|\le k-2$.
By Claim \ref{xi}, $\left|N_C(x)^+\setminus N_C(x_{i+1})\right|\le k-1$. Hence
\begin{align*}
&\quad \left|N_C(x_i^{+|S_i|})\cap N_C(x_{i+1})\cap N_C(x)^+\right|\\
&\ge
\left|N_C(x_i^{+|S_i|})\cap N_C(x)^+\right|+\left| N_C(x_{i+1})\cap N_C(x)^+\right|-\left| N_C(x)^+\right|\\
&\ge t-(k-2)+t-(k-1)-t\\
&\ge 2k-2-(2k-3)= 1.
\end{align*}
Suppose that $\left|N_C(x_i^{+|S_i|})\cap N_C(x_{i+1})\cap N_C(x)^+\right|\ge 2$. Then there exist  $x_\ell^+,x_r^+\in N_C(x_i^{+|S_i|})\cap N_C(x_{i+1})$ with $i+1\le \ell<r\le t+i$. By the definition of $p$ and $q$, we have $p\le \ell<r\le q$, contradicting to the fact that $p\ge q$.
Therefore, we have $\left|N_C(x_i^{+|S_i|})\cap N_C(x_{i+1})\cap N_C(x)^+\right|=1$, implying that
 $t=2k-2$, $\left|N_C(x)^+\cap N_C(x_i^{+|S_i|}) \right|= k$ and  $\left|N_C(x)^+\cap N_C(x_{i+1})\right|=k-1$. Therefore, by the definition of $p$ and $q$ again and the fact that
$p\ge q$, we have $p=q=i+k-1$, $x_{i+1}^+,\dots,x_{i+k-2}^+\notin N_C\left(x_i^{+|S_i|}\right)$ and $x_{i+k}^+,\dots,x_{i+t}^+\notin N_C(x_{i+1})$.
\end{proof}

\begin{claim}\label{evenn}
If $|S_i|$ is even for some $i=1,\dots,t$, then $|S_j|$ is even for $j=i+k-1,\dots,i+2k-2$.
\end{claim}
\begin{proof}
Suppose that $|S_j|$ is odd for some $j=i+k-1,\dots,i+2k-2$. Then $x_j^{+|S_j|}\in S'$.
By Claim \ref{k-2}, $\left|S'\setminus N_C\left(x_i^{+|S_i|}\right)\right|\le k-2$. As $x_{i+1}^+,\dots,x_{i+k-2}^+\notin N_C\left(x_i^{+|S_i|}\right)$ by Claim \ref{even}, we have $S'\setminus \{x_{i+1}^+,\dots,x_{i+k-2}^+\}\subseteq N_C\left(x_i^{+|S_i|}\right)$ and hence $x_i^{+|S_i|}x_j^{+|S_j|}\in E(G)$.
So
\[
x_i^{+|S_i|}x_j^{+|S_j|}\overleftarrow{C}x_{i+1}xx_{j+1}
\overrightarrow{C}x_i^{+|S_i|}
\]
is a cycle of $G$ longer than $C$, a contradiction.
Therefore, $|S_j|$ is even for all $j=i+k-1,\dots,i+2k-2$.
\end{proof}

\begin{claim}\label{odd}
For $i=1,\dots,t$, $|S_i|$ is odd.
\end{claim}
\begin{proof}
Suppose to the contrary that $|S_i|$ is even for some $i=1,\dots,t$.
Then by Claims \ref{even} and \ref{evenn}, $t=2k-2$ and $|S_j|$ is even for $j$ with $i+k-1\le j\le i+2k-2$.
As $|S_{i+k-1}|$ is even, we have by Claim \ref{evenn} that $|S_j|$ is even for  $i+k-1+k-1\le j\le i+k-1+2k-2$ and so $|S_j|$ is even for each $j=1,\dots,2k-2$.

By Claim \ref{xi}, $\left| N_C(x)^+\setminus N_C(x_{i+1})\right|\le k-1$, i.e., $\left| N_C(x)^+\cap N_C(x_{i+1})\right|\ge t-(k-1)= k-1\ge 2$.
Let $x_p^+\in N_C(x)^+\cap N_C(x_{i+1})$ with $p\neq i+1$.
By Claim \ref{even}, $i+2\le p\le i+k-1$. Then by Claims \ref{k-2}  and \ref{even} again, we have $S'\setminus \{x_{p+1}^+,\dots,x_{p+k-2}^+\}\subseteq N_C(x_p^{+|S_p|})$.
As $i+2\le p\le i+k-1$, we have $p+k-2\le i+2k-3$,  $p+1\ge i+3$,  and hence $\{p+1,\dots,p+k-2 \}\subseteq \{i+3,\dots,i+2k-3\}$.
So $x_p^{+|S_p|}x_{i+1}^+\in E(G)$, implying that
\[
x_{i+1}x_p^+\overrightarrow{C}x_p^{+|S_p|}x_{i+1}^+\overrightarrow{C}
x_pxx_{p+1}\overrightarrow{C}x_{i+1}
\]
is a cycle of $G$ longer than $C$, a contradiction.
So $|S_i|$ is odd.
\end{proof}

Recall that $S'=\cup_{i=1}^tS_i'$.
By Claim \ref{odd}, $|S_i|$ is odd for each $i=1,\dots,t$, so
\[
|V(C)|=2|S'|.
\]

By the argument in Claim \ref{Claim4}, we have

\begin{claim} \label{c6}
For any $y\in V(G)\setminus  V(C)$, $N_C(y)\cap S'=\emptyset$.
\end{claim}

By Claims \ref{c2} and \ref{in}, $V(G)\setminus V(C)$ and $S'$ are independent sets.
So $(V(G)\setminus V(C))\cup S'$ is an independent set by Claim \ref{c6}.

Let $S=V(C)\setminus S'$.
Then $|S|=\frac{1}{2}|V(C)|$.
Therefore, $c(G-S)=|V(G)|-|V(C)|+|S'|>\frac{1}{2}|V(C)|$ and so
\[
\tau(G)\le \frac{|S|}{c(G-S)}<1,
\]
contradicting to $\tau(G)\ge 1$.
This completes the proof of Theorem \ref{li}.
\end{proof}

\vspace{5mm}

\noindent {\bf Acknowledgement.}
This work was supported by the National Natural Science Foundation of China (No.~12071158).

%

\end{document}